\documentclass[12pt,reqno,twoside]{amsart}
\usepackage{amssymb,amsfonts,amsthm,amsmath}
\usepackage{cite}
\usepackage[left=1 in,top=1 in,right=1 in,bottom=1 in]{geometry}
\usepackage{enumitem}

\usepackage{hyperref}

\newcommand{\beq}{\begin{equation}}
\newcommand{\eeq}{\end{equation}}
\newcommand{\beqs}{\begin{equation*}}
\newcommand{\eeqs}{\end{equation*}}
\newcommand{\ba}{\begin{array}}
\newcommand{\ea}{\end{array}}
\newcommand{\beas}{\begin{eqnarray*}}
\newcommand{\eeas}{\end{eqnarray*}}
\newcommand{\bea}{\begin{eqnarray}}
\newcommand{\eea}{\end{eqnarray}}
\newcommand{\bal}{\begin{align}}
\newcommand{\eal}{\end{align}}

\newcommand{\bals}{\begin{align*}}
\newcommand{\eals}{\end{align*}}

\newcommand{\R}{\ensuremath{\mathbb R}}

\newcommand{\N}{\ensuremath{\mathbb N}}

\newcommand{\bds}{\begin{displaystyle}}
\newcommand{\eds}{\end{displaystyle}}

\def\eqdef{\stackrel{\rm def}{=}}

\def\varep{\varepsilon}
\def\d{{\rm d}}

\newcommand{\bigo}{\mathcal O}

\newtheorem{theorem}{Theorem}[section]

\newtheorem{lemma}[theorem]{Lemma}

\newtheorem{proposition}[theorem]{Proposition}
\newtheorem{definition}[theorem]{Definition}

\newtheorem*{notation}{Notation}

\theoremstyle{remark}
\newtheorem{remark}[theorem]{\bf{Remark}}

\numberwithin{equation}{section}

\title[Asymptotic expansions for the Lagrangian trajectories]{Asymptotic expansions for the Lagrangian trajectories from solutions of the Navier--Stokes equations}
\author{Luan Hoang}
\address{Department of Mathematics and Statistics, Texas Tech University,  1108 Memorial Circle, Lubbock, TX 79409-1042, U.S.A.}
\email{luan.hoang@ttu.edu}

\dedicatory{\large Dedicated to the memory of Ciprian Foias (1933--2020)}

\date{\today}

\begin{document}
\begin{abstract}
Consider any Leray--Hopf weak solution of the three-dimensional Navier--Stokes equations for incompressible, viscous fluid flows. We prove that any Lagrangian trajectory associated with such a velocity field has an asymptotic expansion, as time tends to infinity, which describes its long-time behavior very precisely.
\end{abstract}

\maketitle

\section{Introduction}\label{intro}
We study the long-time dynamics of the incompressible, viscous fluid flows in the three-dimensional space. 
Theoretically speaking, there are two standard descriptions of fluid flows. One is the Lagrangian that is based on the trajectory $x(t)\in\R^3$ of each initial fluid particle (or material point) $x_0=x(0)$, where $t$ is the time variable. The other is the Eulerian which uses the velocity field $u(x,t)$ and pressure $p(x,t)$, where $x\in\R^3$ is the independent spatial variable representing each fixed position in the fluid. The relation between the two descriptions is the following ordinary differential equations (ODE)
\beq\label{xuode}
 x'=u(x,t).
\eeq
The solutions $x(t)$ of \eqref{xuode} are called the Lagrangian trajectories.

The Eulerian description turns out to be simpler for deriving the set of equations that govern the fluid flows. 
They are called the  Navier--Stokes equations (NSE),
\beq \label{NPDE}
\left \{
\begin{aligned}
 u_t-\nu \Delta u +(u\cdot\nabla )u &=-\nabla p, \\
 {\rm div}\, u &=0.
\end{aligned}
\right.
\eeq 
where $\nu>0$ is the kinematic viscosity, and the unknowns are the velocity $u(x,t)$ and pressure $p(x,t)$. For a solution $(u,p)$, we will conveniently say $u(x,t)$  is a solution of \eqref{NPDE}.

The system \eqref{NPDE} is subject to the initial condition $u(x,0)=u_0(x)$, where $u_0$ is a given initial vector field.

From the mathematical point of view, the NSE is a system of nonlinear partial differential equations, and its understanding is still lacking. Even the basic question about its existence and uniqueness has not been answered completely.
Because of this lack of information about the velocity $u(x,t)$ in \eqref{xuode}, the analysis of the Lagrangian trajectories $x(t)$ is very limited. There have been  results for  the Lagrangian trajectories  in small time intervals. See recent work \cite{CL2019,CKV2016,CVW2015,Sueur2011,GK2019,BF2017,Hernandez2019} and references therein for short-time well-posedness, regularity, and analyticity, based on solutions of the  Euler or Navier--Stokes related systems. See also \cite{MWbook2005} for studies of the topological structures of the flows in the two-dimensional case. 
Naturally, the long-term behavior of the Lagrangian trajectories is even lesser-known.

However, thanks to the remarkable result by C.~Foias and J.-C.~Saut \cite{FS87}, the long-time behavior a solution of the NSE, under the current consideration, can be described completely, bypassing its yet unknown global well-posedness. 
This is part of the Foias--Saut theory of asymptotic expansions and their associated normal form and nonlinear spectral  manifolds for the NSE. See their work \cite{FS84a,FS84b,FS87,FS91}, which were developed further or extended in \cite{FHOZ1,FHOZ2,FHS1,HM1,HM2,CaH1,CaH2,HTi1,Shi2000,Minea}.
(The interested reader is referred to \cite{FHS2} for a brief survey on the subject.)  
The goal of this paper is to investigate \eqref{xuode} in this direction in order to gain knowledge of precise long-time dynamics of fluid flows in the Lagrangian description.

First, we recall the type of asymptotic expansions, as time tends to infinity, studied here as well as in previous work.

\begin{definition}\label{Sexpand}
Let $(X,\|\cdot\|_X)$ be a normed space and  $(\alpha_n)_{n=1}^\infty$ be a  sequence of strictly increasing non-negative real numbers. 
A function $f:[T,\infty)\to X$, for some $T\ge 0$, is said to have an asymptotic expansion
\beq\label{ex1}
f(t)\sim\sum_{n=1}^\infty f_n(t)e^{-\alpha_n t} \quad\text{in } X,
\eeq 
where each $f_n:\R\to X$ is a polynomial, if one has, for any $N\ge 1$, that
\beq\label{ex2}
\Big \|f(t)-\sum_{n=1}^N f_n(t)e^{-\alpha_n t}\Big \|_X=\bigo(e^{-(\alpha_N+\varepsilon_N)t})\text{ as }t\to\infty,
\eeq
for some $\varepsilon_N>0$.
\end{definition}

Clearly, if \eqref{ex2} holds for some $N\in\N=\{1,2,3,\ldots\}$ and some polynomials $f_n$'s, for $n=1,\ldots,N$, then those $f_n$'s are unique.
Consequently, the polynomials $f_n$'s, for all $n\in\N$, in Definition \ref{Sexpand} are unique.
Moreover, in the case $X$ is finite dimensional, all of its norms are equivalent, hence, the expansion \eqref{ex1} is the same for any norm on $X$.
Also, in many cases, the norm $\|\cdot\|_X$ is a standard and well-known one, hence will be implicitly understood.

\medskip
Returning to the NSE, it was proved by Foias and Saut \cite{FS87} that any solution $u(x,t)$ of the NSE processes an asymptotic expansion of type \eqref{ex1}. Our goal is to establish the same result for solutions of \eqref{xuode}, where $u(x,t)$ is a Leray--Hopf weak solution of the NSE \eqref{NPDE}. Indeed, we prove in Theorems \ref{Thm1} and \ref{Thm3} that when $u(x,t)$ satisfies the no-slip boundary condition, or is a spatially periodic solution with zero average, then system \eqref{xuode} has a solution $x(t)$, for sufficiently large $t$, which admits an asymptotic expansion in $\R^3$. The starting point is a simple realization in Proposition \ref{yzerolem} that each trajectory $x(t)$ converges exponentially, as $t\to\infty$. The general case of spatially periodic solutions  is treated in section \ref{sec3}.

Our obtained results give very precise long-time dynamics for the Lagrangian trajectories for general weak solutions of the NSE.
They contrast with the papers cited above which only yield short-time properties. Moreover, our approach draws strong conclusions with  relatively simple proofs.

The rest of this section is focused on preliminaries. We consider the NSE \eqref{NPDE} in one of the following two specified situations.

\textbf{Dirichlet boundary condition (DBC).} 
Let $\Omega$ be an bounded, open, connected set in $\R^3$ with $C^\infty$ boundary. We consider \eqref{NPDE} in $\Omega\times(0,\infty)$ with the boundary condition $u=0$ on $\partial \Omega\times(0,\infty)$.

\textbf{Spatial periodicity condition (SPC).} Fix a vector $\mathbf L=(L_1,L_2,L_3)\in (0,\infty)^3$. We consider \eqref{NPDE} in $\R^3\times(0,\infty)$ with $u(\cdot,t)$ and $p(\cdot,t)$ being $\mathbf L$-periodic for $t>0$.

Here, a function $g$ defined on $\R^3$ is called $\mathbf L$-periodic if
\beqs
g(x+L_i e_i)=g(x) \text{ for $i=1,2,3$ and all $x\in \R^3$,}
\eeqs
where $\{e_1,e_2,e_3\}$ is the standard canonical basis of $\R^3$.

Define domain $\Omega=(0,L_1)\times(0,L_2)\times (0,L_3)$ in this case. A function $g$ is said to have zero average over $\Omega$ if
\beqs
\int_\Omega g(x) \d x =0.
\eeqs

We recall some needed basic elements from the theory of the NSE. For details, the reader is referred to the books \cite{CFBook, TemamAMSbook, TemamSIAMbook, FMRTBook}.
Below, $H^m=W^{m,2}$, for $m\in\N$, denote the standard Sobolev spaces. 

In the (DBC) case, let $\mathcal V$ be the set of divergence-free vector fields in  $C_c^\infty(\Omega)^3$.
Define $\mathcal X$ to be the set of functions in $\bigcap_{m=1}^\infty H^m(\Omega)^3$ that are divergence-free and vanish on the boundary $\partial \Omega$, and denote $\Omega^*=\bar \Omega$. Let $\mathbb L^2(\Omega)=L^2(\Omega)^3$ and $\mathbb H^m(\Omega)=H^m(\Omega)^3$. 

In the (SPC) case, let $\mathcal{V}$ be the set of $\mathbf L$-periodic trigonometric polynomial vector fields on $\R^3$ which are divergence-free and  have zero average over $\Omega$. 
Define $\mathcal X=\mathcal V$, and denote $\Omega^*=\R^3$.
Let $\mathbb L^2(\Omega)$ (respectively, $\mathbb H^m(\Omega)$) be the space of $\mathbf L$-periodic vector fields on $\R^3$ that belong to $L_{\rm loc}^2(\R^3)^3$ (respectively, $H_{\rm loc}^m(\R^3)^3$), and is endowed with the inner product and norm of  $L^2(\Omega)^3$ (respectively, $H^m(\Omega)^3$).

In both cases, define space $H$ (respectively, $V$) to be the closure of $\mathcal V$ in $\mathbb L^2(\Omega)$ (respectively, $\mathbb H^1(\Omega)$).
The Leray projection $\mathbb P$ is the orthogonal projection from $\mathbb L^2(\Omega)$ to $H$.
The Stokes operator is $(-\mathbb P\Delta)$ defined on $V\cap \mathbb H^2(\Omega)$.

Denote the spectrum of Stokes operator by $\{ \Lambda_k:k\in\N \}$, where $\Lambda_k$'s are positive, strictly increasing to infinity.
Let $\mathcal S$ be the additive semigroup generated by $\nu\Lambda_k$'s, that is, 
\beqs
\mathcal S=\Big\{\nu\sum_{j=1}^N \Lambda_{k_j}:N,k_1,\ldots,k_N\in\N \Big\}.
\eeqs

We arrange the set $\mathcal S$ as a sequence $(\mu_n)_{n=1}^\infty$ of positive, strictly increasing numbers.
Clearly,  
\beq\label{limu} 
\lim_{n\to\infty} \mu_n=\infty,
\eeq
\beq\label{semiS}
\mu_n+\mu_k\in \mathcal S\quad\forall n,k\in\N.
\eeq

For convenience, we will write $f(t)=g(t)+\bigo(h(t))$ to indicate
$$|f(t)-g(t)|=\bigo(h(t))\text{ as }t\to\infty.$$  

\section{Main results}\label{sec2}

For any $u_0\in H$, there exists a Leray--Hopf weak solution $u(x,t)$ of \eqref{NPDE} on $[0,\infty)$ with initial condition $u(x,0)=u_0(x)$. By its eventual regularity, there is $T\ge 0$ such that
$u\in C^\infty(\Omega^*\times [T,\infty))$ and satisfies the corresponding (DBC) or (SPC).

\textbf{(A)} \emph{Throughout this section, let us fix such a Leray--Hopf weak solution $u(x,t)$ and a Lagrangian trajectory 
$x(t)\in C^1([T,\infty),\Omega)$ in the (DBC) case, or $x(t)\in C^1([T,\infty),\R^3)$ in the (SPC) case. 
}

A discussion about assumption (A) is given in Remark \ref{inrmk} below.

It is proved in   \cite{FS87} that the solution $u(x,t)$ has an asymptotic expansion, in the sense of Definition \ref{Sexpand},
\beq \label{expand}
u(\cdot,t) \sim \sum_{n=1}^\infty q_n(\cdot,t)e^{-\mu_n t}\text{ in }\mathbb H^m(\Omega),
\eeq
for any $m\in\N$, where $q_j(\cdot,t)$'s are polynomials in $t$ 
with values in $\mathcal X\subset C^\infty(\Omega^*)^3$.

One can write each polynomial $q_n(x,t)$, for $n\ge 1$, explicitly as
\beq\label{qx}
q_n(x,t)=\sum_{k=0}^{d_n} t^k q_{n,k}(x),\text{ where }d_n\ge 0,\text{ and } q_{n,k}\in\mathcal X.
\eeq

In fact, $q_1(x,t)$ is independent of $t$, hence we write
\beq\label{q1ind}
q_1(x,t)=q_1(x)\in \mathcal X. 
\eeq

According to the expansion \eqref{expand} with $m=2$ and Definition \ref{Sexpand}, we have 
\beqs
\Big\|u(\cdot,t)- \sum_{n=1}^N q_n(\cdot,t)e^{-\mu_n t}\Big\|_{H^2(\Omega)^3}=\bigo(e^{-(\mu_N+\delta_N)t}),
\eeqs
for any $N\in\N$, and some $\delta_N>0$.

By Morrey's embedding theorem, it follows that
\beq\label{usup}
\sup_{x\in\Omega^*}\Big|u(x,t)- \sum_{n=1}^N q_n(x,t)e^{-\mu_n t}\Big|=\bigo(e^{-(\mu_N+\delta_N)t}).
\eeq

In particular, letting $N=1$ in \eqref{usup} and using the fact \eqref{q1ind}, we infer 
\beqs 
\sup_{x\in\Omega^*}|u(x,t)|\le \sup_{x\in\Omega^*}|q_1(x)| e^{-\mu_1 t}+\bigo(e^{-(\mu_1+\delta_1)t})=\bigo(e^{-\mu_1 t}).
\eeqs 
Therefore, there is $C_0>0$ such that 
\beq\label{udec}
\sup_{x\in\Omega^*}|u(x,t)|\le C_0e^{-\mu_1 t}\text{ for all } t\ge T.
\eeq

Taking $x=x(t)$ in \eqref{usup} and \eqref{udec}, one has
\beq\label{ut}
\Big |u(x(t),t)- \sum_{n=1}^N q_n(x(t),t)e^{-\mu_n t}\Big |=\bigo(e^{-(\mu_N+\delta_N)t}),
\eeq
\beq\label{utt}
|u(x(t),t)|\le C_0e^{-\mu_1 t}\text{ for all } t\ge T.
\eeq
 
\begin{proposition}\label{yzerolem}
The limit $x_*\eqdef \lim_{t\to\infty}x(t)$  exists and belongs to $\Omega^*$,
and  
 \beq\label{er0}
|x(t)-x_*|=\bigo(e^{-\mu_1 t}).
\eeq
\end{proposition}
\begin{proof}
For $t\ge T$, we have
\beq\label{xt}
x(t)=x(T)+\int_T^t u(x(\tau),\tau)\d\tau.
\eeq

It follows \eqref{xt} and estimate \eqref{utt} that
\beq\label{xstar}
x_*=\lim_{t\to\infty}x(t)=x(T)+\int_T^\infty u(x(\tau),\tau)\d\tau \text{ which exists in } \R^3.
\eeq
Obviously, $x_*\in\Omega^*$. 
By \eqref{xt}, \eqref{xstar}, and \eqref{utt} again, we obtain, for $t\ge T$, 
\beqs
|x(t)-x_*|= \Big|\int_t^\infty u(x(\tau),\tau)\d\tau\Big|\le \int_t^\infty C_0e^{-\mu_1\tau}\d\tau =C_0\mu_1^{-1}e^{-\mu_1 t},
\eeqs
which proves \eqref{er0}.
\end{proof}

\medskip
\begin{notation}
For $x\in \R^3$, denote $x^{(0)}=1$, and by $x^{(k)}$ the $k$-tuple $(x,\ldots,x)$ for $k\ge 1$.

If $m\in \N$ and $\mathcal L$ is an $m$-linear mapping from $(\R^3)^m$ to $\R^3$, the norm of $\mathcal L$ is defined by
\beqs
\|\mathcal L\|=\max\{ |\mathcal L(y_1,y_2,\ldots,y_m)|:y_j\in\R^3,|y_j|=1,\text{ for } 1\le j\le m\}.
\eeqs 
\end{notation}

It is known that the norm $\|\mathcal L\|$ belongs to  $[0,\infty)$, and one has 
\beq\label{multiL}
|\mathcal L(y_1,y_2,\ldots,y_m)|\le \|\mathcal L\|\cdot |y_1|\cdot |y_2|\cdots |y_m|
\quad \forall y_1,y_2,\ldots,y_m\in\R^n.
\eeq

Below are consequences of expansion \eqref{expand} and Proposition \ref{yzerolem}.
Let $x_*$ be as in Proposition \ref{yzerolem}. 

\textbf{Consideration I:} The (SPC) case, or $x_*\in\Omega$ in the (DBC) case. 

\textbf{Consideration II:}  The (DBC) case with $x_*\in\partial \Omega$. 

We focus on Consideration I first.
By using the Taylor expansion for each $q_{n,k}(x)$, see e.g. \cite[Chapter XVI, \S 6]{LangI1968}, we obtain, for any $s\ge 0$, 
\beq\label{qTay}
q_{n,k}(x)=\sum_{m=0}^s \frac1{m!}D_x^m q_{n,k}(x_*)(x-x_*)^{(m)} + g_{n,k,s}(x),
\eeq
where  $D_x^m q_{n,k}$ denotes the $m$-th order derivative of $q_{n,k}$, 
and $g_{n,k,s}\in C(\Omega^*)^3$ satisfying
\beq\label{gx}
g_{n,k,s}(x)=\bigo(|x-x_*|^{s+1}) \text{ as }x\to x_*.
\eeq
Here, $D_x^m q_{n,k}$ is $q_{n,k}$ for $m=0$, and is an $m$-linear mapping from $(\R^3)^m$ to $\R^3$, for $m\ge 1$.

Substituting \eqref{qTay} into \eqref{qx} gives
\beqs
q_n(x,t)=\sum_{k=0}^{d_n} t^k \Big[ \sum_{m=0}^s \frac1{m!}D_x^m q_{n,k}(x_*)(x-x_*)^{(m)} + g_{n,k,s}(x)\Big],
\eeqs
which can be rewritten as
\beq\label{qnxap}
q_n(x,t)=\sum_{m=0}^s \mathcal Q_{n,m}(x_*,t)(x-x_*)^{(m)}+\sum_{k=0}^{d_n} t^k g_{n,k,s}(x),
\eeq
where 
\beq\label{Qnm}
\mathcal Q_{n,m}(x_*,t)=\sum_{k=0}^{d_n}  \frac{t^k}{m!}D_x^m q_{n,k}(x_*)=\frac{1}{m!}D_x^m q_n(x_*,t).
\eeq 

In particular,  
\beq\label{QD}
\mathcal Q_{n,0}(x_*,t)=q_n(x_*,t),\quad \mathcal Q_{n,1}(x_*,t)=D_x q_n(x_*,t),\quad \mathcal Q_{n,2}(x_*,t)=\frac12D_x^2 q_n(x_*,t).
\eeq

Note from \eqref{Qnm} that  $\mathcal Q_{n,m}(x_*,t)$ is a polynomial in $t$ valued in the space of $m$-linear mappings from $(\R^3)^m$ to $\R^3$. Therefore, one has, for any $k\ge 1$ and $m\ge 0$,  
\beq\label{Qze2}
\|\mathcal Q_{k,m}(x_*,t)\|=\bigo(e^{\delta t})\quad \forall\delta>0.
\eeq

Denote $z(t)=x(t)-x_*$. Then \eqref{er0} reads as
\beq\label{zest} 
|z(t)|=\bigo(e^{-\mu_1 t}).
\eeq 

Combining \eqref{qnxap} for $x=x(t)$ with \eqref{gx} and \eqref{zest} yields 
\beqs
q_n(x(t),t)=\sum_{m=0}^s \mathcal Q_{n,m}(x_*,t)z(t)^{(m)}+\sum_{k=0}^{d_n} t^k \bigo(e^{-\mu_1(s+1)t}),
\eeqs
thus 
\beq\label{qntap}
q_n(x(t),t)=\sum_{m=0}^s \mathcal Q_{n,m}(x_*,t)z(t)^{(m)}+\bigo(e^{-(\mu_1(s+1)-\delta)t}) \quad \forall \delta>0.
\eeq

Our main result is the following.

\begin{theorem}\label{Thm1}
Under Consideration I, there exist polynomials $\zeta_n:\R\to\R^3$, for $n\ge 0$, such that solution $x(t)$ has an asymptotic expansion, in the sense of Definition \ref{Sexpand},
 \beq\label{xex}
 x(t)\sim x_*+\sum_{n=1}^\infty \zeta_n(t) e^{-\mu_n t} \text{ in } \R^3,
 \eeq
where each $\zeta_n$, for $n\ge 1$,  is the unique polynomial solution of the following differential equation
\beq\label{zeneq}
\zeta_n'(t)-\mu_n\zeta_n(t) = \sum_{\mu_k+\mu_{j_1}+\mu_{j_2}+\ldots+\mu_{j_m}=\mu_n} \mathcal Q_{k,m}(x_*,t) (\zeta_{j_1}(t),\ldots,\zeta_{j_m}(t)). 
\eeq
for all  $t\in\R$. More explicitly, $\zeta_n(t)$ can be calculated recursively by formula \eqref{zenx} below.
\end{theorem}

Before proving Theorem \ref{Thm1}, we explain the formulas appearing there.

\begin{enumerate}[label={\rm (\alph*)}]
 \item\label{xpa} Formula \eqref{zeneq} is the concise form of the following 
\beq\label{zeo}
\zeta_n'(t)-\mu_n\zeta_n(t) = q_n(x_*,t) + \sum_{\substack{m, k,j_1,\ldots,j_m\ge 1,\\ \mu_k+\mu_{j_1}+\mu_{j_2}+\ldots+\mu_{j_m}=\mu_n}} \mathcal Q_{k,m}(x_*,t) (\zeta_{j_1}(t),\ldots,\zeta_{j_m}(t))
\eeq

Indeed, when $m=0$, the indices $j_1,\ldots,j_m$, numbers $\mu_{j_1},\ldots,\mu_{j_m}$, and functions $\zeta_{j_1}(t),\ldots,\zeta_{j_m}(t)$ are not needed in \eqref{zeneq}, then $\mu_k=\mu_n$, which implies $k=n$, and $\mathcal Q_{k,m}(x_*,t) (\zeta_{j_1}(t),\ldots,\zeta_{j_m}(t))$ is just $q_n(x_*,t)$. 
When $m\ge 1$, the indices $j_1,\ldots,j_m$ are present in \eqref{zeneq} and are in $\N$. 

\item\label{fin} The sum on the right-hand side of \eqref{zeneq} is a finite one.
Indeed, when $m=0$, one has, again, $k=n$.
Consider $m\ge 1$. Since $\mu_k,\mu_{j_1},\ldots,\mu_{j_m}\ge \mu_1$, we have 
$$(1+m)\mu_1\le  \mu_k+\mu_{j_1}+\mu_{j_2}+\ldots+\mu_{j_m}=\mu_n,$$
$$m_k,\mu_{j_1},\ldots,\mu_{j_m}<\mu_n.$$
Thus, $m\le \mu_n/\mu_1-1$ and  $k,j_1,\ldots,j_m\le n-1$. (Observe that $k\le n$ for both cases of $m$.)
Therefore, we can explicitly rewrite \eqref{zeneq}, via \eqref{zeo}, as
\beq\label{zep}
\zeta_n'(t)-\mu_n\zeta_n(t) = q_n(x_*,t) + \sum_{m=1}^{s_n}\sum_{\substack{k,j_1,\ldots,j_m= 1,\\ \mu_k+\mu_{j_1}+\mu_{j_2}+\ldots+\mu_{j_m}=\mu_n}}^{n-1} \mathcal Q_{k,m}(x_*,t) (\zeta_{j_1}(t),\ldots,\zeta_{j_m}(t)),
\eeq
where $s_n=\min\{ s\in \N:s\ge \mu_n/\mu_1-1\}$. 

\item\label{xpn} To give examples, we write equation \eqref{zeneq} explicitly, by using the identities in \eqref{QD},  for $n=1,2,3$ as
\begin{align}
\zeta_1'(t)-\mu_1 \zeta_1(t)&=q_1(x_*),\label{ze1}\\
\zeta_2'(t)-\mu_2 \zeta_2(t)&=D_x q_1(x_*) \zeta_1(t)  + q_2(x_*,t),\notag \\
\zeta_3'(t)-\mu_3 \zeta_3(t)&= \frac12D_x^2q_1(x_*)(\zeta_1(t),\zeta_1(t))+ D_x q_2(x_*,t) \zeta_1(t)  + q_3(x_*,t).\notag
\end{align}

\item Equation \eqref{zeneq} comes from the following approximation lemma. It is the particular Case (iii) of \cite[Lemma 4.2]{HM2}, which essentially originates from Foias--Saut's work \cite{FS87}.
\end{enumerate}

\begin{lemma}\label{polylem}
Let $(X,\|\cdot\|_X)$ be a Banach space.
Let $p:\R\to X$ be a polynomial, and $g\in C([t_*,\infty),X)$, for some $t_*\ge 0$, satisfy 
 \beqs
 \|g(t)\|_X\le Me^{-\delta t} \quad \forall t\ge t_*, \quad \text{for some } M,\delta>0.
 \eeqs

Let $\gamma$ be a positive real number. Suppose  that $y\in C([t_*,\infty),X)\cap C^1((t_*,\infty),X)$ solves the equation
 \beqs
 y'(t)-\gamma y(t) =p(t)+g(t) \quad  \text{for }t > t_*,
 \eeqs
and satisfies
\beq\label{yexpdec}
  \lim_{t\to\infty} (e^{-\gamma t}\|y(t)\|_X)=0.
\eeq
  
Then there exists a unique polynomial $q:\R\to X$ such that 
 \beq\label{gball}
  \|y(t)-q(t)\|_X\le \frac{M}{\gamma+\delta}e^{-\delta t} \quad \text{ for all }t\ge t_*.
  \eeq

More precisely, $q(t)$ is the unique polynomial solution of
\beq\label{pode2}
q'(t)-\gamma q(t) = p(t)\quad \text{for }t\in \R,
\eeq 
and can be explicitly defined by
\beq\label{qdef}        
 q(t)=-\int_t^\infty e^{\gamma(t-\tau)}p(\tau) \d\tau.
\eeq
\end{lemma}

In fact, the statements in \cite[Lemma 4.2]{HM2} are for $t_*=0$, but they can be easily generalized and proved for any $t_*\ge 0$, see e.g. \cite[Lemma 2.2]{CaH3}.

\begin{enumerate}[label={\rm (\alph*)}]  \setcounter{enumi}{4}
\item\label{xpq} For each $n\in\N$, the right-hand side of \eqref{zeneq} is an $\R^3$-valued polynomial.
Then polynomial solution $\zeta_n$ exists and is unique, see \eqref{pode2} and \eqref{qdef} in Lemma \ref{polylem}.
Explicitly, for $n\ge 1$ and $t\in\R$,
\begin{multline}\label{zenx}
\zeta_n(t)=-\int_t^\infty e^{\mu_n (t-\tau)}\Big\{ q_n(x_*,\tau) \\
 + \sum_{m=1}^{s_n}\sum_{\substack{k,j_1,\ldots,j_m= 1,\\ \mu_k+\mu_{j_1}+\mu_{j_2}+\ldots+\mu_{j_m}=\mu_n}}^{n-1} \mathcal Q_{k,m}(x_*,\tau) (\zeta_{j_1}(\tau),\ldots,\zeta_{j_m}(\tau))\Big\}\d \tau.
\end{multline}

In particular, when $n=1$, one has
\beq\label{zq1} \zeta_1(t)=-q_1(x_*)/\mu_1\text{ for }t\in\R.
\eeq 

\item We explicate \ref{fin} above. One has, for each $n\ge 1$, and integers $M\ge \mu_n/\mu_1-1$, $K\ge n$, $J\ge n-1$, that 
\begin{multline}\label{sumequiv}
\sum_{\mu_k+\mu_{j_1}+\mu_{j_2}+\ldots+\mu_{j_m}=\mu_n}
\mathcal Q_{k,m}(x_*,t) (\zeta_{j_1}(t),\ldots,\zeta_{j_m}(t)) \\
=\sum_{m=0}^M \sum_{k=1}^K \sum_{\substack{j_1,\ldots,j_m=1,\\ \mu_k+\mu_{j_1}+\mu_{j_2}+\ldots+\mu_{j_m}=\mu_n}}^J 
\mathcal Q_{k,m}(x_*,t) (\zeta_{j_1}(t),\ldots,\zeta_{j_m}(t)),
\end{multline}
with the right-hand side being interpreted for $m=0$ in the same way as in \ref{xpa} above.

Indeed, the right-hand side of \eqref{sumequiv} is part of the left-hand side. Reversely, the arguments in \ref{fin} show that the left-hand side of \eqref{sumequiv} equals the right-hand side of \eqref{zep}, which, in turn, is part of the right-hand side of \eqref{sumequiv}. Hence, both sides of \eqref{sumequiv} must be the same.
\end{enumerate}

\begin{proof}[Proof of Theorem \ref{Thm1}]
The proof follows the general scheme of Foias--Saut \cite{FS87} with simplified presentation as in \cite{HM2}.
By the virtue of \eqref{er0}, it suffices to prove that
\beqs
z(t)=x(t)-x_*\sim \sum_{n=1}^\infty \zeta_n(t) e^{-\mu_n t}.
\eeqs

To this end, we will prove, by induction, that given any $N\in \N$, there exists $\varep_N>0$ such that
\beq\label{erN}
\Big|z(t)-\sum_{n=1}^N \zeta_n(t) e^{-\mu_n t}\Big|=\bigo(e^{-(\mu_N+\varep_N)t}).
\eeq

\medskip
Consider $N=1$.  We have from \eqref{ut} that
\beqs
z'(t)=x'(t)=u(x(t),t)=q_1(x(t))e^{-\mu_1 t}+\bigo(e^{-(\mu_1+\delta_1)t}). 
\eeqs

Writing $q_1(x(t))$ by \eqref{qntap} for $n=1$, $s=0$ and $\delta=\mu_1/2$, we obtain 
\beqs
z'(t)=[q_1(x_*) + \bigo(e^{-\mu_1 t/2})] e^{-\mu_1 t}  +\bigo(e^{-(\mu_1+\delta_1)t})=q_1(x_*)e^{-\mu_1 t}  +\bigo(e^{-(\mu_1+\varep_1)t}), 
\eeqs
where $\varep_1=\min\{\mu_1/2,\delta_1\}$.
Let $w_0(t)=e^{\mu_1 t} z(t)$.
Then
\beqs
w_0'(t)=\mu_1 e^{\mu_1 t} z(t)+ e^{\mu_1 t} z'(t)=\mu_1 w_0(t) + q_1(x_*) + \bigo(e^{-\varep_1 t}).
\eeqs
Thus,
\beq\label{wzero}
w_0'(t)-\mu_1 w_0(t)=q_1(x_*)+ \bigo(e^{-\varep_1 t}).
\eeq

We apply Lemma \ref{polylem} to equation \eqref{wzero}, that is, $t_*=T$, $y(t)=w_0(t)$, $\gamma=\mu_1$, $p(t)=q_1(x_*)$, $\delta=\varep_1$, and note that 
\beqs
\lim_{t\to\infty}(e^{-\gamma t} |y(t)|)=\lim_{t\to\infty} |z(t)|=0,
\eeqs
hence condition \eqref{yexpdec} is met.
It follows \eqref{gball} that 
\beq\label{wzeroap}
|w_0(t)-\zeta_1(t)|=\bigo(e^{-\varep_1 t}),
\eeq
where $\zeta_1(t)$ is the unique polynomial solution of \eqref{ze1}.
Multiplying \eqref{wzeroap} by $e^{-\mu_1 t}$, we obtain \eqref{erN} for $N=1$.

\medskip
Now, given $N\in \N$, assume \eqref{erN} holds with $\zeta_n$ being the unique polynomial solutions of \eqref{zeneq} for $n=1,\ldots, N$.
Let $z_N(t)=\sum_{n=1}^N \zeta_n(t) e^{-\mu_n t}$ and $\tilde z_N(t)=z(t)-z_N(t)$.

Since $\zeta_1(t)$ is a constant vector, see \eqref{zq1}, we have from the definition of $z_N(t)$ that 
\beq\label{zNe} 
|z_N(t)|=\bigo(e^{-\mu_1 t}).
\eeq 
Also, estimate \eqref{erN} reads as
\beq\label{zerN}
|\tilde z_N(t)|=\bigo(e^{-(\mu_N+\varep_N)t}).
\eeq

Define $w_N(t)=e^{\mu_{N+1}t}\tilde z_N(t)$. 
Taking derivative of $w_N(t)$ gives
\beq\label{wNp}
 w_N' =\mu_{N+1}w_N+ e^{\mu_{N+1}t}\Big(z'-\sum_{n=1}^N e^{-\mu_n t}(\zeta_n'-n\zeta_n)\Big).
\eeq

We will find an appropriate expansion for $z'(t)$ in \eqref{wNp}. By \eqref{ut}, we have
\beqs
z'(t)=x'(t)=u(x(t),t)
= \sum_{k=1}^{N+1} q_k(x(t),t)e^{-\mu_k t} +\bigo(e^{-(\mu_{N+1}+\delta_{N+1})t}).
\eeqs 

We make use of the approximation \eqref{qntap} for each $q_k(x(t),t)$, for $k=1,2,\ldots,N+1$, with $s=s_{N+1}$ defined in \eqref{zep}, and $\delta=\mu_1/2$. Noticing that  $\mu_1(s_{N+1}+1)\ge \mu_{N+1}$, we obtain   
\begin{align*}
z'(t)
&=\sum_{k=1}^{N+1} \sum_{m=0}^{s_{N+1}} \mathcal Q_{k,m}(x_*,t)z(t)^{(m)} e^{-\mu_k t}
+\sum_{k=1}^{N+1}\bigo(e^{-(\mu_{N+1}-\mu_1/2) t})e^{-\mu_k t} +\bigo(e^{-(\mu_{N+1}+\delta_{N+1})t}).
\end{align*}

For the middle sum on the right-hand side, we estimate $\mu_k-\mu_1/2\ge \mu_1-\mu_1/2=\mu_1/2$. Then it follows that
\beq\label{zz}
z'(t)=\sum_{k=1}^{N+1} \sum_{m=0}^{s_{N+1}} \mathcal Q_{k,m}(x_*,t)z(t)^{(m)} e^{-\mu_k t} +\bigo(e^{-(\mu_{N+1}+\widehat\delta_{N+1})t}),
\eeq 
where $\widehat\delta_{N+1}=\min\{\mu_1/2,\delta_{N+1}\}$.

Denote $Q=\mathcal Q_{k,m}(x_*,t)$ in the calculations below. For $m,k\ge 1$, we write
\beqs 
Qz(t)^{(m)}
=Q(z_N(t)+\tilde z_N(t),\ldots, z_N(t)+\tilde z_N(t))=Qz_N(t)^{(m)}+\sum Q(y_1,\ldots,y_m),
\eeqs 
where the last sum is a finite sum, and the vectors $y_1,\ldots,y_m$ belong to $\{z_N(t),\tilde z_N(t)\}$ with at least one of them being $\tilde z_N(t)$.
We estimate each $ Q(y_1,\ldots,y_m)$ by \eqref{multiL}, with $\|Q\|$ being bounded by \eqref{Qze2} for $\delta=\varep_N/2$, and 
use estimates \eqref{zNe},  \eqref{zerN} for $|z_N(t)|$, $|\tilde z_N(t)|$, respectively. We obtain
\begin{align*}
Qz(t)^{(m)}
&=Qz_N(t)^{(m)}+\bigo_m(e^{(\varep_N/2)t} \cdot e^{-(\mu_N+\varep_N)t})\\
&=Q\Big(\sum_{j_1=1}^N \zeta_{j_1}e^{-\mu_{j_1} t},\sum_{j_2=1}^N \zeta_{j_2}e^{-\mu_{j_2} t},\ldots, \sum_{j_m=1}^N \zeta_{j_m}e^{-\mu_{j_m} t}\Big)+\bigo(e^{-(\mu_N+\varep_N/2)t})\\
&=\sum_{j_1,\ldots,j_m=1}^N Q(\zeta_{j_1},\zeta_{j_2},\ldots, \zeta_{j_m}) e^{-(\mu_{j_1}+\ldots+\mu_{j_m})t}+\bigo(e^{-(\mu_N+\varep_N/2)t}). 
\end{align*}
Combining this with \eqref{zz}, we have
\begin{align*}
z'(t)&=\sum_{k=1}^{N+1} \sum_{m=0}^{s_{N+1}} \sum_{j_1,\ldots,j_m=1}^N \mathcal Q_{k,m}(x_*,t)(\zeta_{j_1},\ldots,\zeta_{j_m}) e^{-(\mu_k+\mu_{j_1}+\ldots+\mu_{j_m})t}\\
&\quad +\sum_{k=1}^{N+1} \sum_{m=1}^{s_{N+1}}(e^{-(\mu_N+\varep_N/2)t}))e^{-\mu_k t}+\bigo(e^{-(\mu_{N+1}+\widehat\delta_{N+1})t}).
\end{align*}

In the middle terms on the right-hand side,  the number $\mu_N+\mu_k$ is in $\mathcal S$, which is due to \eqref{semiS}, greater than $\mu_N$,  and hence, greater or equal to $\mu_{N+1}$.
Therefore,
\beq\label{zpo}
z'(t)=\sum_{k=1}^{N+1} \sum_{m=0}^{s_{N+1}} \sum_{j_1,\ldots,j_m=1}^N \mathcal Q_{k,m}(x_*,t)(\zeta_{j_1},\ldots,\zeta_{j_m}) e^{-(\mu_k+\mu_{j_1}+\ldots+\mu_{j_m})t}+\bigo(e^{-(\mu_{N+1}+\varep'_{N+1})t}),
\eeq
where $\varep'_{N+1}=\min\{\varep_N/2,\widehat\delta_{N+1}\}$.

Note that each $\mathcal Q_{k,m}(x_*,t) (\zeta_{j_1}(t),\ldots,\zeta_{j_m}(t))$ is a $\R^3$-valued polynomial, hence
\beq\label{Qze1}
|\mathcal Q_{k,m}(x_*,t) (\zeta_{j_1}(t),\ldots,\zeta_{j_m}(t))|=\bigo(e^{\delta t})\quad \forall\delta>0,
\eeq

For the first sums on the right-hand side of \eqref{zpo}, one has, thanks to \eqref{semiS}, $\mu_k+\mu_{j_1}+\ldots+\mu_{j_m}\in \mathcal S$, hence there is a unique $n\in\N$ such that 
\beqs
\mu_n=\mu_k+\mu_{j_1}+\ldots+\mu_{j_m}.
\eeqs 

We split the sums into two parts: $I_1(t)$ corresponds to $n\le N+1$, and $I_2(t)$ corresponds to $n\ge N+2$. Clearly, by using the estimate \eqref{Qze1}, we have $I_2(t)=\bigo(e^{-(\mu_{N+2}-\delta)t})$, for any $\delta>0$.
Taking $\delta=\widehat\varep_{N+1}\eqdef (\mu_{N+2}-\mu_{N+1})/2$ gives 
\beqs
I_2(t)=\bigo(e^{-(\mu_{N+1}+\widehat\varep_{N+1})t})=\bigo(e^{-(\mu_{N+1}+\varep_{N+1})t}),
\eeqs
where $\varep_{N+1}=\min\{\varep'_{N+1},\widehat\varep_{N+1}\}$.
We also rewrite
$I_1(t)=\sum_{n=1}^{N+1} J_n(t) e^{-\mu_n t}$,
where 
\beqs
J_n(t)=\sum_{k=1}^{N+1}\sum_{m=0}^{s_{N+1}} \sum_{\substack{j_1,\ldots,j_m=1,\\ \mu_k+\mu_{j_1}+\mu_{j_2}+\ldots+\mu_{j_m}=\mu_n}}^N \mathcal Q_{k,m}(x_*,t) (\zeta_{j_1}(t),\ldots,\zeta_{j_m}(t)), \quad \text{for }n=1,\ldots,N+1 .
\eeqs

We obtain
\beq\label{zp}
z'(t)=\sum_{n=1}^{N+1} J_n(t) e^{-\mu_n t}+\bigo(e^{-(\mu_{N+1}+\varep_{N+1})t}).
\eeq

Combining \eqref{wNp} with \eqref{zp} gives
\beq\label{wN0}
 w_N' 
 =\mu_{N+1}w_N +e^{\mu_{N+1}t}\sum_{n=1}^N e^{-\mu_n t}\Big\{ J_n- (\zeta_n'-n\zeta)\Big\} +J_{N+1} +\bigo(e^{-\varep_{N+1} t}).
\eeq 

For $n=1,\ldots,N+1$, applying formula \eqref{sumequiv} to $M=s_{N+1}\ge \mu_{N+1}/\mu_1-1\ge \mu_n/\mu_1-1$, $K=N+1\ge n$, and $J=N\ge n-1$, we have
\beq\label{Jn}
J_n(t)=\sum_{\mu_k+\mu_{j_1}+\mu_{j_2}+\ldots+\mu_{j_m}=\mu_n} \mathcal Q_{k,m}(x_*,t) (\zeta_{j_1}(t),\ldots,\zeta_{j_m}(t)) .
\eeq

For $n=1,\ldots, N$, it follows relation \eqref{Jn} and equation \eqref{zeneq} that  $J_n-(\zeta_n'-n\zeta)=0$. Thus, equation \eqref{wN0} yields
\beq\label{wNeq}
 w_N' -\mu_{N+1}w_N 
 =J_{N+1} +\bigo(e^{-\varep_{N+1} t}).
\eeq

We apply Lemma  \ref{polylem} to equation \eqref{wNeq}, i.e., $t_*=T$, $y(t)=w_N(t)$, $\gamma=\mu_{N+1}$, 
$$p(t)= J_{N+1}(t)=\sum_{\mu_k+\mu_{j_1}+\mu_{j_2}+\ldots+\mu_{j_m}=\mu_{N+1}} \mathcal Q_{k,m}(x_*,t) (\zeta_{j_1}(t),\ldots,\zeta_{j_m}(t)),$$ 
and $\delta=\varep_{N+1}$. 
Note that 
\beqs
\lim_{t\to\infty} (e^{-\gamma t}|y(t)|)=\lim_{t\to\infty}|\tilde z_N(t)|=0,
\eeqs
which verifies condition \eqref{yexpdec}.
Then one has 
\beq\label{wN1}
\Big|w_N(t)-\zeta_{N+1}(t)\Big|=\bigo(e^{-\varep_{N+1} t}),
\eeq
where $\zeta_{N+1}:\R\to\R^3$ is the unique polynomial solution of equation \eqref{zeneq} for $n=N+1$.

Multiplying \eqref{wN1} by $e^{-\mu_{N+1}t}$ gives
\beqs
\Big|\tilde z_N(t)-\zeta_{N+1}(t)e^{-\mu_{N+1} t}\Big|=\bigo(e^{-(\mu_{N+1}+\varep_{N+1}) t}),
\eeqs
which proves \eqref{erN} for $N:=N+1$. 

\medskip
By the induction principle, \eqref{erN} holds for all integers $N\in \N$.
The proof is complete.
\end{proof}

Next is the result corresponding to Consideration II.

\begin{theorem}\label{Thm3}
 Under Consideration II, one has
 \beq\label{zc2}
 |x(t)-x_*|=\bigo(e^{-\mu t}) \text{ for all } \mu>0.
 \eeq
\end{theorem}
\begin{proof}
Let $N_*>1$ be any integer, denote $s_*=s_{N_*}\in\N$ which is defined in \eqref{zep}. 
For $1\le n\le N_*$ and $0\le k\le d_n$, there exists an extension $Q_{n,k}\in C_c^{s_*+1}(\R^3)^3$ such that
$q_{n,k}=Q_{n,k}\big|_{\bar\Omega}$. Define $Q_n(x,t)=\sum_{k=0}^{d_n} t^k Q_{n,k}(x)$. 
Then \eqref{qTay} and  \eqref{Qnm} hold true with $D_x^m q_{n,k}$ being replaced with $D_x^m Q_{n,k}$, and $D_x^m q_n(x,t)$ with $D_x^m Q_n(x,t)$.
Repeat the above proof of Theorem \ref{Thm1} with finite induction, we assert that \eqref{erN} holds for $N=N_*$, i.e.,
 \beq\label{zlast}
 |z(t)-\sum_{n=1}^{N_*} \zeta_n(t)e^{-\mu_n t}|=\bigo(e^{-(\mu_{N_*}+\varep_{N_*}) t}).
 \eeq
 
Since $x_*\in\partial\Omega$ and $q_{n,k}\in\mathcal X$, one has $q_{n,k}(x_*)=0$, hence $q_n(x_*,t)=0$ for $1\le n\le N_*$. 
By \eqref{zq1}, $\zeta_1=0$. One can verify by the recursive formula \eqref{zenx} that $\zeta_n=0$ for $n=1,\ldots, N_*$. 
Combining this with \eqref{zlast} yields 
 \beqs
 |z(t)|=\bigo(e^{-\mu_{N_*} t}) \text{ for any integer } N_*>1.
 \eeqs
By letting $N_*\to\infty$ and using property \eqref{limu}, we obtain \eqref{zc2}.
 \end{proof}

\begin{remark}\label{inrmk}
This is a discussion about assumption (A).

In the (SPC) case, by its regularity and spatial periodicity, $u$ is bounded on $\R^3\times[T,T']$ for any $T'>T$.
Then, given any $x_0\in\R^3$, there is a unique solution $x(t)\in C^1([T,\infty),\R^3)$ of \eqref{xuode} with $x(T)=x_0$. 

In the (DBC) case, given any $x_0\in\Omega$, there exists a unique solution $x(t)\in C^1([T,T_{\rm max}),\Omega)$ of \eqref{xuode} with $x(T)=x_0$, where $[T,T_{\rm max})$ is the maximal interval of existence. In case $T_{\rm max}<\infty$,  we have, by the boundedness of $u$ on $\Omega\times [T,T_{\max })$, 
\beqs
\lim_{t\to T_{\rm max}^-} x(t)=\bar x\eqdef x(T)+\int_T^{T_{\rm max}} u(x(\tau),\tau)\d\tau,
\eeqs
which exists and must belong to $\partial \Omega$. Then $u(\bar x,t)=0$, and by defining $x(t)=\bar x$ for $t\ge T_{\rm max}$ we obtain a solution $x(t)\in \bar \Omega$ of \eqref{xuode} for all $t\ge T$. However, its long-time behavior is not interesting. Therefore, in assumption (A) above, we only consider $T_{\rm max}=\infty$, that is, $x(t)\in\Omega$ for all $t\ge T$. 
\end{remark}

\section{General spatial periodicity case}\label{sec3}
We consider the general case of (SPC), when the velocity field is not required to have zero average over $\Omega$.

Let $u(x,t)\in C_{x,t}^{2,1}(\R^3\times (0,\infty))\cap C(\R^3\times[0,\infty))$ and $p(x,t)\in C_{x}^{1}(\R^3\times (0,\infty))$
be $\mathbf L$-periodic  functions that form a solution $(u,p)$ of the NSE \eqref{NPDE}.

Let   $x(t)\in\R^3$ be a solution of \eqref{xuode} on $(0,\infty)$.
The next theorem shows that $x(t)$ has a similar asymptotic expansion to \eqref{xex} in Theorem \ref{Thm1}.

\begin{theorem}\label{Thm2}
There exist $x_*\in\R^3$ and polynomials $X_n:\R\to\R^3$, for $n\in\N$, such that
\beq\label{xnew}
x(t)\sim (x_*+U_0 t) +\sum_{n=1}^\infty X_n(t)e^{-\mu_n t} \text{ in }\R^3,
\eeq
where  $U_0=(L_1L_2L_3)^{-1}\int_{\Omega} u(x,0)\d x$.

\end{theorem}
\begin{proof}
We use the standard Galilean transformation. Set 
\beqs
v(X,t)=u(X+U_0 t,t)-U_0\text{ and } P(X,t) = p(X + U_0 t,t),\quad X\in\R^3, \ t\ge 0.
\eeqs 

One can verify that  $(v,P)$ is a classical solution of the NSE \eqref{NPDE} on $\R^3\times (0,\infty)$.
Moreover,  $(v,P)$ is $\mathbf L$-periodic, and $v(\cdot,t)$ has zero average for each $t\ge 0$. 

Let $X(t)=x(t)-U_0t$. We have
\beqs
X'(t)=x'(t)-U_0=u(x(t),t)-U_0=v(x(t)-U_0 t,t)+U_0-U_0= v(X(t),t).
\eeqs

Applying Theorem \ref{Thm1} to $v(X,t)$ and $X(t)$ yields
\beqs
X(t)\sim x_*+\sum_{n=1}^\infty X_n(t)e^{-\mu_n t},
\eeqs 
where $x_*$ is a constant vector in $\R^3$, and $X_n$'s are $\R^3$-valued polynomials on $\R$.
Consequently, we obtain
\beqs
x(t)=X(t)+U_0 t\sim x_*+U_0t +\sum_{n=1}^\infty X_n(t)e^{-\mu_n t},
\eeqs
which proves \eqref{xnew}.
\end{proof}

\begin{remark}
We end the paper with the following comments.

\begin{enumerate}[label=\rm (\alph*)]
\item The asymptotic expansion is uniquely determined for each given Leray--Hopf weak solution, but does not imply the uniqueness of the  Leray--Hopf weak solutions starting from the same initial condition.

\item The expansion of $x(t)$ in Theorem \ref{Thm1} is not driven by a dissipative ODE, but rather by the exponential decay in the time-dependent expansion \eqref{expand} of $u(x,t)$.
This is different from the previous results obtained for the NSE \cite{FS87,HM2,CaH1,CaH2,HTi1} or general nonlinear differential equations considered in \cite{Shi2000,CaH3,Minea}. 

\item The above proof of Theorem \ref{Thm1} can be adapted to study the Navier--Stokes equations in different contexts when $u(x,t)$ processes an asymptotic expansion similar to \eqref{expand} such as those in \cite{HM2,CaH1,HTi1}.

\item Although the above presentation is focused on the three-dimensional space, the  obtained results are equally true for the two-dimensional case.  Moreover, Theorem \ref{Thm1} is not restricted to just the velocity field and space $\R^3$, but in fact, applies to general differential equations in Banach spaces. 
\end{enumerate}
\end{remark}


\begin{thebibliography}{10}

\bibitem{BF2017}
N.~Besse and U.~Frisch.
\newblock A constructive approach to regularity of {L}agrangian trajectories
  for incompressible {E}uler flow in a bounded domain.
\newblock {\em Comm. Math. Phys.}, 351(2):689--707, 2017.

\bibitem{GK2019}
G.~Camliyurt and I.~Kukavica.
\newblock On the {L}agrangian and {E}ulerian analyticity for the {E}uler
  equations.
\newblock {\em Phys. D}, 376/377:121--130, 2018.

\bibitem{CaH2}
D.~Cao and L.~Hoang.
\newblock Asymptotic expansions in a general system of decaying functions for
  solutions of the {N}avier-{S}tokes equations.
\newblock {\em Annali di Matematica Pura ed Applicata}, 3(199):1023--1072,
  2020.

\bibitem{CaH3}
D.~Cao and L.~Hoang.
\newblock Asymptotic expansions with exponential, power, and logarithmic
  functions for non-autonomous nonlinear differential equations.
\newblock {\em Journal of Evolution Equations}, pages 1--45, 2020.
\newblock In press, DOI:10.1007/s00028-020-00622-w. Preprint
  https://arxiv.org/abs/1911.11077.

\bibitem{CaH1}
D.~Cao and L.~Hoang.
\newblock Long-time asymptotic expansions for {N}avier--{S}tokes equations with
  power-decaying forces.
\newblock {\em Proceedings of the Royal Society of Edinburgh: Section A
  Mathematics}, 150(2):569--606, 2020.

\bibitem{CFBook}
P.~Constantin and C.~Foias.
\newblock {\em {Navier-{S}tokes equations}}.
\newblock {Chicago Lectures in Mathematics}. University of Chicago Press,
  Chicago, IL, 1988.

\bibitem{CKV2016}
P.~Constantin, I.~Kukavica, and V.~Vicol.
\newblock Contrast between {L}agrangian and {E}ulerian analytic regularity
  properties of {E}uler equations.
\newblock {\em Ann. Inst. H. Poincar{\'e} Anal. Non Lin{\'e}aire},
  33(6):1569--1588, 2016.

\bibitem{CL2019}
P.~Constantin and J.~La.
\newblock Note on {L}agrangian-{E}ulerian methods for uniqueness in
  hydrodynamic systems.
\newblock {\em Adv. Math.}, 345:27--52, 2019.

\bibitem{CVW2015}
P.~Constantin, V.~Vicol, and J.~Wu.
\newblock Analyticity of {L}agrangian trajectories for well posed inviscid
  incompressible fluid models.
\newblock {\em Adv. Math.}, 285:352--393, 2015.

\bibitem{FHOZ1}
C.~Foias, L.~Hoang, E.~Olson, and M.~Ziane.
\newblock {On the solutions to the normal form of the {N}avier-{S}tokes
  equations}.
\newblock {\em Indiana Univ. Math. J.}, 55(2):631--686, 2006.

\bibitem{FHOZ2}
C.~Foias, L.~Hoang, E.~Olson, and M.~Ziane.
\newblock {The normal form of the {N}avier-{S}tokes equations in suitable
  normed spaces}.
\newblock {\em Ann. Inst. H. Poincar{\'e} Anal. Non Lin{\'e}aire},
  26(5):1635--1673, 2009.

\bibitem{FHS1}
C.~Foias, L.~Hoang, and J.-C. Saut.
\newblock {Asymptotic integration of {N}avier-{S}tokes equations with potential
  forces. {II}. {A}n explicit {P}oincar{\'e}-{D}ulac normal form}.
\newblock {\em J. Funct. Anal.}, 260(10):3007--3035, 2011.

\bibitem{FHS2}
C.~Foias, L.~Hoang, and J.-C. Saut.
\newblock {N}avier and {S}tokes meet {P}oincar{\'e} and {D}ulac.
\newblock {\em J. Appl. Anal. Comput.}, 8(3):727--763, 2018.

\bibitem{FMRTBook}
C.~Foias, O.~Manley, R.~Rosa, and R.~Temam.
\newblock {\em {Navier-{S}tokes equations and turbulence}}, volume~83 of {\em
  {Encyclopedia of Mathematics and its Applications}}.
\newblock Cambridge University Press, Cambridge, 2001.

\bibitem{FS84a}
C.~Foias and J.-C. Saut.
\newblock {Asymptotic behavior, as {$t\rightarrow +\infty $}, of solutions of
  {N}avier-{S}tokes equations and nonlinear spectral manifolds}.
\newblock {\em Indiana Univ. Math. J.}, 33(3):459--477, 1984.

\bibitem{FS84b}
C.~Foias and J.-C. Saut.
\newblock {On the smoothness of the nonlinear spectral manifolds associated to
  the {N}avier-{S}tokes equations}.
\newblock {\em Indiana Univ. Math. J.}, 33(6):911--926, 1984.

\bibitem{FS87}
C.~Foias and J.-C. Saut.
\newblock {Linearization and normal form of the {N}avier-{S}tokes equations
  with potential forces}.
\newblock {\em Ann. Inst. H. Poincar{\'e} Anal. Non Lin{\'e}aire}, 4(1):1--47,
  1987.

\bibitem{FS91}
C.~Foias and J.-C. Saut.
\newblock {Asymptotic integration of {N}avier-{S}tokes equations with potential
  forces. {I}}.
\newblock {\em Indiana Univ. Math. J.}, 40(1):305--320, 1991.

\bibitem{Hernandez2019}
M.~Hernandez.
\newblock Mechanisms of {L}agrangian analyticity in fluids.
\newblock {\em Arch. Ration. Mech. Anal.}, 233(2):513--598, 2019.

\bibitem{HM1}
L.~T. Hoang and V.~R. Martinez.
\newblock Asymptotic expansion in {G}evrey spaces for solutions of
  {N}avier-{S}tokes equations.
\newblock {\em Asymptot. Anal.}, 104(3--4):167--190, 2017.

\bibitem{HM2}
L.~T. Hoang and V.~R. Martinez.
\newblock Asymptotic expansion for solutions of the {N}avier-{S}tokes equations
  with non-potential body forces.
\newblock {\em J. Math. Anal. Appl.}, 462(1):84--113, 2018.

\bibitem{HTi1}
L.~T. Hoang and E.~S. Titi.
\newblock Asymptotic expansions in time for rotating incompressible viscous
  fluids.
\newblock {\em Annales de l'Institut Henri Poincar{\'e}. Analyse Non
  Lin{\'e}aire}, pages 1--29, 2020.
\newblock In press, DOI:10.1016/j.anihpc.2020.06.005.

\bibitem{LangI1968}
S.~Lang.
\newblock {\em Analysis {I}}.
\newblock Addison-Wesley Publishing Company, 1968.

\bibitem{MWbook2005}
T.~Ma and S.~Wang.
\newblock {\em Geometric theory of incompressible flows with applications to
  fluid dynamics}, volume 119 of {\em Mathematical Surveys and Monographs}.
\newblock American Mathematical Society, Providence, RI, 2005.

\bibitem{Minea}
G.~Minea.
\newblock {Investigation of the {F}oias-{S}aut normalization in the
  finite-dimensional case}.
\newblock {\em J. Dynam. Differential Equations}, 10(1):189--207, 1998.

\bibitem{Shi2000}
Y.~Shi.
\newblock A {F}oias-{S}aut type of expansion for dissipative wave equations.
\newblock {\em Comm. Partial Differential Equations}, 25(11-12):2287--2331,
  2000.

\bibitem{Sueur2011}
F.~Sueur.
\newblock Smoothness of the trajectories of ideal fluid particles with
  {Y}udovich vorticities in a planar bounded domain.
\newblock {\em J. Differential Equations}, 251(12):3421--3449, 2011.

\bibitem{TemamSIAMbook}
R.~Temam.
\newblock {\em {Navier-{S}tokes equations and nonlinear functional analysis}},
  volume~66 of {\em {CBMS-NSF Regional Conference Series in Applied
  Mathematics}}.
\newblock Society for Industrial and Applied Mathematics (SIAM), Philadelphia,
  PA, second edition, 1995.

\bibitem{TemamAMSbook}
R.~Temam.
\newblock {\em {Navier-{S}tokes equations}}.
\newblock AMS Chelsea Publishing, Providence, RI, 2001.
\newblock Theory and numerical analysis, Reprint of the 1984 edition.

\end{thebibliography}
\def\cprime{$'$}

\end{document}